\documentclass{article}
\usepackage[english]{babel}
\usepackage[T1]{fontenc}
\usepackage[utf8]{inputenc}

\usepackage[top=2cm,bottom=3cm,left=3cm,right=3cm,marginparwidth=1.75cm]{geometry}

\usepackage{lmodern}
\usepackage{microtype}

\usepackage{graphicx}\graphicspath{{img/}}
\usepackage{xcolor}

\usepackage{amssymb,amsmath,amsthm,mathtools,cases}
\usepackage{enumitem}
\usepackage{bm}

\usepackage{comment}

\usepackage{caption}
\usepackage{psfrag}
\usepackage{tikz}

\usepackage[active]{srcltx}   
\usepackage[colorlinks=true,urlcolor=magenta!60!black,citecolor=green!50!black,linkcolor=red!90!black,bookmarksnumbered=true,pdfauthor={Jeremie Bettinelli,Dimitri Korkotashvili}]{hyperref}

\numberwithin{equation}{section}

\theoremstyle{plain}
\newtheorem{theorem}{Theorem}

\newtheorem{proposition}[theorem]{Proposition}
\newtheorem{corollary}[theorem]{Corollary}

\theoremstyle{definition}
\newtheorem{remark}{Remark}

\newcommand{\eps}{\varepsilon}
\newcommand{\de}{\mathrel{\mathop:}\hspace*{-.6pt}=}

\newcommand{\sew}{\,{\Join}\,}

\newcommand{\ba}{\bm{a}}
\newcommand{\bc}{\mathbf{c}}
\newcommand{\m}{\mathbf{m}}
\newcommand{\bp}{\mathbf{p}}
\newcommand{\bt}{\mathbf{t}}
\newcommand{\bu}{\mathbf{u}}
\newcommand{\balpha}{\boldsymbol{\alpha}}
\newcommand{\bbeta}{\boldsymbol{\beta}}
\newcommand{\bgamma}{\boldsymbol{\gamma}}
\newcommand{\blambda}{\boldsymbol{\lambda}}
\newcommand{\bmu}{\boldsymbol{\mu}}
\newcommand{\bwp}{\boldsymbol{\wp}}

\newcommand{\N}{\mathbb{N}}

\newcommand{\overl}[1]{\mkern 3mu\overline{\mkern-3mu#1\mkern-1mu}\mkern 1mu}
\newcommand{\bipsymb}{%
	\begin{tikzpicture}
	\draw (.03,0)--(.1,0); \draw (0,0) circle (.8pt); \filldraw (.13,0) circle (.8pt);
	\end{tikzpicture}}
\newcommand{\bip}[2]{\smash{\rlap{\hspace*{#1}\raisebox{-.3em}{\bipsymb}}#2}}

\newcommand{\cM}{\mathcal{M}}
\newcommand{\bcM}{\bip{.3em}{\cM}}
\newcommand{\ccM}{\smash{\overl{\mathcal{M}}}}
\newcommand{\bccM}{\bip{.3em}{\ccM}}
\newcommand{\cS}{\mathcal{S}}

\newcommand{\cU}{\mathcal{U}}
\newcommand{\bcU}{\bip{0em}{\cU}}

\definecolor{alpha}{RGB}{227,0,0}
\definecolor{beta}{RGB}{170,85,255}
\definecolor{mu}{RGB}{0,199,0}
\definecolor{stepcol}{rgb}{0,.44,.22}

\usepackage{contour}
\contourlength{0.6pt}
\newcommand{\unli}[1]{\underline{\phantom{\smash{\textbf{#1}}}}\llap{\contour{white}{\textbf{#1}}}}

\makeatletter
\renewcommand\paragraph{\@startsection{paragraph}{4}{\z@}%
                                    {3.25ex \@plus1ex \@minus.2ex}%
                                    {-1em}%
                                    {\normalfont\normalsize\bfseries\unli}}
\makeatother

\title{Link between bipartite and general unicellular toroidal maps via slit--slide--sew bijections}
\author{J\'er\'emie Bettinelli\thanks{LIX, \'Ecole polytechnique, CNRS, Institut Polytechnique de Paris. Partially supported by ANR-21-CE48-0007 \emph{IsOMa}, ANR-23-CE48-0018 \emph{CartesEtPlus}, ANR-24-CE40-7809 \emph{LOUCCOUM}.}%
	\and %
Dimitri Korkotashvili\thanks{LIX, \'Ecole polytechnique, CNRS, Institut Polytechnique de Paris.}}

\let\oldemph\emph
\renewcommand{\emph}[1]{\textcolor{red!65!black}{\oldemph{#1}}}

\begin{document}
\maketitle

\begin{abstract}
We relate general maps to bipartite maps through a bijection of type slit-slide-sew. We provide an involution on arbitrary genus maps with even degree faces. This enables a full interpretation of the relation between general and bipartite maps, in the case of genus $1$ maps with a unique face.

The main tool is the use of rotations along well-chosen specific loops. Once a noncontractible simple loop is given, one slits along it, slides one notch, and sews back. This mildly modifies the structure of the map along the loop, changing the parity of the length of other loops crossing it. In the unicellular toroidal setting, the structure of noncontractible loops is simple enough to enable a full correspondence between general and bipartite maps.
\end{abstract}

\section{Introduction}

In the last decade, so-called \emph{slit-slide-sew bijections} have been introduced in~\cite{Bet14} in order to explain combinatorial identities on forests and on plane quadrangulations with a boundary. They were then developed to broader of similar contexts: plane bipartite and quasibipartite maps~\cite{Bet20}, quadratic formulas for plane maps derived from the KP hierarchy~\cite{Lou19}, plane tight maps~\cite{BoGuMi24}, plane constellations and quasiconstellations~\cite{BeKo26}, plane maps whose edges are oriented~\cite{BeFuLo24arX}, plane bipartite maps with prescribed degrees~\cite{Sch25arX}.

These bijections interpret combinatorial identities on several classes of (possibly decorated) maps, which may usually be derived from known counting results. They however provide alternate proofs to those formulas and, usually, enable to recover them from easy to obtain initial conditions.

Planarity plays a crucial role in all the aforementioned works. In the present paper, we consider maps of positive genus, mainly toroidal maps. We will also restrict our attention in due course to unicellular maps, that is, maps with a unique face. Although the structure of these maps is simpler, they are, from a combinatorial point of view, quite rich, and have attracted a lot of attention in the past~\cite{Cha10,Cha11,BeCh11,Ber12,ChFeFu13}.

\paragraph*{Maps.} In the present work, we define a \emph{map} as a cellular embedding of a finite graph (possibly with multiple edges and loops) into a compact orientable surface, considered up to homeomorphisms. Cellular means that the connected components of the complement of edges, called \emph{faces}, are homeomorphic to 2-dimensional open disks. Furthermore, one particular corner is distinguished and called the \emph{root} of the map. We insist on the fact that, throughout this work, all maps are considered to be rooted. See Figure~\ref{map}.

\begin{figure}[ht!]
	\centering\includegraphics[width=.5\linewidth]{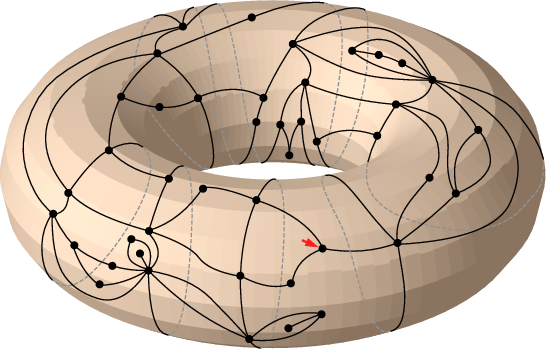}
	\caption{Map of genus~$1$ with even degree faces. The root is the corner represented by the red arrow.}
	\label{map}
\end{figure}

The vertex incident to the root of a map is called the \emph{root vertex}. The genus of the underlying surface is called the \emph{genus} of the map. Distances in a map are always understood as graph distances of the underlying graph.

We denote by~$E(\m)$ and~$V(\m)$ the edge-set and vertex-set of a map~$\m$. We call\footnote{Beware that the notion of \emph{half-edge} also has another definition in the literature.} \emph{half-edge} an edge given with one of its two possible orientations; a half-edge~$e$ is thus oriented from its \emph{tail}, denoted by~$e^-$, to its \emph{head}, denoted by~$e^+$.

\paragraph*{Jackson--Visentin identity.}
We investigate a special case of the Jackson--Visentin identity, relating, on the one side, maps of some given genus~$g$ having controlled even face degrees with, on the other side, \textbf{bipartite} maps with same face degree constraints and genuses ranging from~$0$ to~$g$, where each drop in genus is ``compensated'' by two extra marks on vertices. More precisely, let us fix for now
\begin{itemize}
	\item a genus $g\ge 0$,
	\item a number $n\ge 1$ of edges,
	\item and a positive even integers set~$\Delta\subseteq 2\N$ of admissible degrees.
\end{itemize}
We let $\prescript{\Delta\!}{n}\cM^{k}$ be the set of maps of genus~$k$ with~$n$ edges, whose faces all have degree lying in~$\Delta$. We also let~$\prescript{\Delta\!}{n}\bcM^{k}_{2i}$ be the set of maps from~$\prescript{\Delta\!}{n}\cM^{k}$ whose vertices are properly\footnote{Remember that a coloring is called \emph{proper} when no edges link two vertices of the same color.} colored black or white, in such a way that the root vertex is white, and such that~$2i$ distinct black vertices are marked. Observe that, if such a bicoloring exists, it is unique: a vertex is colored white if and only if it lies at even distance from the root vertex. Moreover, the existence of such a bicoloring is equivalent to the absence of odd-length cycles: in this case, the map is called \emph{bipartite}. Consequently, $\prescript{\Delta\!}{n}\bcM^{k}_{0}$ is the set of maps from~$\prescript{\Delta\!}{n}\cM^{k}$ that are bipartite. For instance, the map depicted in Figure~\ref{map} is not bipartite since there are odd-length cycles.

Denoting cardinalities with $\lvert\cdot\rvert$, it holds \cite[Theorem~4.1]{JaVi99} (see also~\cite{Fan14} for this precise reformulation) that
\begin{equation}\label{eqJV}
\bigl\lvert\prescript{\Delta\!}{n}\cM^{g}\bigr\rvert = \sum_{i=0}^{g} 4^{g-i}\,\bigl\lvert\prescript{\Delta\!}{n}\bcM^{g-i}_{2i}\bigr\rvert.
\end{equation}
This is actually a generalization of the particular case $\Delta=\{4\}$, which had been obtained earlier~\cite{JaVi90} and is dubbed the \emph{quadrangulation relation}.

\paragraph*{Unicellular toroidal maps.}
In the present paper, we focus on the particular case of \emph{unicellular} maps, that is, maps having a single face, and we restrict ourselves to \emph{toroidal} maps, that is, of genus~$1$. This amounts to take $n\ge 1$, fix $g=1$ and $\Delta=\{2n\}$. Equation~\eqref{eqJV} thus reads 
\begin{equation}\label{equnicel}
\bigl\lvert\,\cU^{1}\bigr\rvert = 4\,\bigl\lvert\,\bcU^{1}\bigr\rvert + \bigl\lvert\,\bcU^{0}_{2}\bigr\rvert,
\end{equation}
where
\begin{itemize}
	\item $\cU^{1}$ is the set of unicellular maps of genus~$1$ with~$n$ edges;
	\item $\bcU^{1}$ is the set of \textbf{bipartite} unicellular maps of genus~$1$ with~$n$ edges;
	\item $\bcU^{0}_{2}$ is the set of plane trees  with~$n$ edges and~$2$ distinct marked vertices at odd distance from the root vertex.
\end{itemize}

Equation~\eqref{equnicel} is the one we will bijectively interpret in the present paper through so-called \emph{slit-slide-sew} operations. The reader already familiar with this kind of techniques may note that, in contrast with previous works, all the degrees are preserved; there will be no ``degree transfer'' here.

\paragraph*{Covered maps.}
Our approach extends to general maps with prescribed face degrees (as in~\cite{Bet20}) that carry a distinguished covering unicellular map.

In this context, the faces of the maps are labeled~$f_1$, \dots, $f_r$. The \emph{type} of a map is the tuple $(\deg(f_1),\dots,\deg(f_r))$ of the degrees of its faces. In particular, the unicellular maps considered above are the maps of type $(2n)$, that is, with a unique face, $f_1$, of degree~$2n$.

A \emph{covered map} is a pair $(\m,\bu)$ of maps with same genus, same vertex-set $V(\bu)=V(\m)$, same root, and such that~$\bu$ is unicellular and made of edges from~$\m$, that is, $E(\bu)\subseteq E(\m)$. By \oldemph{same root}, we mean that the distinguished corner of~$\m$ is included in that of~$\bu$. The first coordinate, $\m$, will be called the \emph{ambient map}, while the second coordinate, $\bu$, will be called the \emph{covering map}.

In the planar case (genus~$0$), these are so-called \emph{tree-rooted maps}. By duality, covered maps are actually, in any genus, closely related to tree-rooted maps, which have been the focus of several studies \cite{Mul67,WaLe72I,WaLe72II,WaLe72III,Ber07,BeCh11HZ,FrSe20,AlFuSa24,FrSe25}.

For a tuple $\ba=(a_1,\ldots,a_r)$ of positive even integers, we denote by~$\ccM^{1}(\ba)$ the set of covered maps of genus~$1$ whose ambient map is of type~$\ba$. We also let $\bccM^{1}(\ba)\subseteq \ccM^{1}(\ba)$ be the subset of covered maps whose ambient map is bipartite, and $\bccM_{2}^{0}(\ba)$ the set of covered maps of genus~$0$ whose ambient map is of type~$\ba$, carrying two marked vertices at odd distance from the root vertex. Note that maps of type~$\ba$ in genus~$0$ are necessarily bipartite. Furthermore, in a covered map whose ambient map is bipartite, the parity of distances in the covering map or in the ambient map are the same.

Our bijections enable to obtain the following relation.

\begin{proposition}\label{propcm}
It holds that
\begin{equation}\label{eqcm}
\bigl\lvert\ccM^{1}(\ba)\bigr\rvert =  4\,\bigl\lvert\bccM^{1}(\ba)\bigr\rvert + \bigl\lvert\bccM_{2}^{0}(\ba)\bigr\rvert.
\end{equation}
\end{proposition}

We will prove this in Section~\ref{seccov} while performing our slit-slide-sew operation on the covering unicellular map. When the ambient map is unicellular, which amounts to take $\ba=(2n)$, we immediately recover~\eqref{equnicel} since there is a unique choice for the covering map in this case: the ambient map itself.

\paragraph*{Future perspectives.}
The present work is a first step toward bijectively recovering~\eqref{eqJV}, as well as a prescribed face degree analog. Natural intermediate results under study include general genus~$1$ maps or unicellular maps in genus~$2$ or more.

\section{Main involution on general maps with a distinguished loop}

At this stage, we work with general maps of any positive genus~$g$. We denote by~$\cM^{g}\de\bigcup_n\prescript{2\N\!}{n}\cM^{g}$ the set of maps of genus~$g$ with even face degrees. For a tuple $\ba=(a_1,\ldots,a_r)$ of positive even integers, we denote by~$\cM^{g}(\ba)$ the set of maps of genus~$g$ and type~$\ba$, that is, whose faces are labeled~$f_1$, \dots, $f_r$ and such that $(\deg(f_1),\dots,\deg(f_r))=\ba$.

\subsection{Preliminary definitions}

\paragraph*{Paths and loops.}
A \emph{path} is a finite sequence $\bwp=(e_1,e_2,\dots,e_k)$ of half-edges such that $e_i^+=e_{i+1}^-$, for $1\le i \le k-1$. Its \emph{length} is the integer~$k$, which we denote by $[\bwp]\de k$. The \emph{reverse} of~$\bwp$ is the path $\bar\bwp=(\bar e_k,\dots,\bar e_2,\bar e_1)$, where~$\bar e_i$ is the reverse of~$e_i$, that is, the half-edge corresponding to the same edge, with reverse orientation.

We say that~$\bwp$ \emph{links} the vertices~$e_1^-$ to~$e_k^+$. A path is called \emph{simple} if the vertices it visits are all distinct. If the paths $\bwp_1=(e_1,e_2,\dots,e_k)$ and $\bwp_2=(e_{k+1},e_{k+2},\dots,e_m)$ satisfy $e_k^{+}=e_{k+1}^{-}$, we denote their \emph{concatenation} by $\bwp_1\bullet \bwp_2\de (e_1,e_2,\dots,e_m)$.

A \emph{loop} is a path $\bgamma=(e_1,e_2,\dots,e_k)$ linking some vertex to itself, that is, such that $e_1^{-} = e_k^{+}$. The loop $\bgamma=(e_1,e_2,\dots,e_k)$ is \emph{simple} if the path $(e_1,e_2,\dots,e_{k-1})$ is simple; in other words, all the vertices it visits are visited exactly once, except the origin and end, which is visited exactly twice.

Through the embedding of the graph in the definition of a map, any half-edge~$e$ is given with a continuous mapping $s\in [0,1]\mapsto e(s)$ with $e(0)=e^-$ and $e(1)=e^+$. A loop $\bgamma=(e_1,e_2,\dots,e_k)$ is \emph{contractible} (resp.\ \emph{noncontractible}), if the curve $t\in(0,k]\mapsto e_{\lceil t \rceil}(t+1-\lceil t \rceil)$ associated with it is homotopic (resp.\ nonhomotopic) to a single point.

\paragraph*{Slit-slide-sew operation.}
The techniques of so-called \emph{slit-slide-sew} have been used -- exclusively in the planar setting -- in order to mildly transform maps \cite{Bet14,Lou19,Bet20,BoGuMi24,Sch25arX}, or extensions thereof \cite{BeKo26,BeFuLo24arX}, by sliding along a well-chosen path. 

The general idea is as follows. Let us assume that we are given a simple path $\bwp=(e_1,e_2,\dots,e_k)$ in some map~$\m$. We may slit the surface into which the map is embedded along this path. In the map~$\m$, this doubles the path~$\bwp$, making up two copies: one to its left, $\bwp^\ell=(e^\ell_1,e^\ell_2,\dots,e^\ell_k)$, and one to its right, $\bwp^r=(e^r_1,e^r_2,\dots,e^r_k)$. The way each extremity of the path is treated has to be made precise: we may keep~$\bwp^\ell$ and~$\bwp^r$ connected at the extremal vertex, or disconnect them by sliding inside some incident corner.

We then sew back~$\bwp^\ell$ onto~$\bwp^r$ but only after sliding by one unit, in the sense that we match~$e^\ell_{i}$ with~$e^r_{i\pm 1}$. For further reference, we denote by $e^\ell_{i}\sew e^r_{i\pm 1}$ the edge resulting of this gluing. Again, what to do exactly at the extremities of the paths has to be set. This toy procedure possibly affects only up to two faces, those into which the extremities of the paths~$\bwp^\ell$ and~$\bwp^r$ have been possibly disconnected. The rest of the embedding, as well as the genus, remains unchanged.

Conceptually, this operation defines a function that maps the original map to a new map, whose parameters have been mildly changed. To promote this construction to a \oldemph{bijection}, the construction has to be revertible: from the transformed map, one should be able to recover the image of the sliding path~$\bwp$ in order to slide backwards. This necessitates the introduction of a canonical notion of sliding path, which can be tracked down in the resulting map.

\subsection{Involution}\label{secinvol}

We present here the involution at the heart of our processes. It works on general maps, possibly with prescribed face degrees, but who carry a distinguished noncontractible simple loop. It can be thought of as an \href{https://www.francocube.com/notation}{R or R' move} on a genus~$g$ equivalent of a \oldemph{Rubik's cube}. Basically, we will rotate, in one direction or another, by one notch along the distinguished loop. See Figure~\ref{involfig}.

Let $\m\in\cM^{g}$, let $\bgamma=(\gamma_1,\dots,\gamma_k)$ be a noncontractible simple loop in~$\m$, and let $\eps\in\{+,-\}$.

\begin{description}[style=nextline]
\item[\textcolor{stepcol}{1. Slit}]\hypertarget{slitstep}
	We cut the surface along~$\bgamma$ and obtain a surface with two boundary components, one carrying the left copy $\bgamma^\ell=(\gamma^\ell_1,\dots,\gamma^\ell_k)$ of~$\bgamma$, and one carrying its right copy $\bgamma^r=(\gamma^r_1,\dots,\gamma^r_k)$. 

\item[\textcolor{stepcol}{2. Slide and Sew}]\hypertarget{sewstep}
	We sew back~$\bgamma^\ell$ onto~$\bgamma^r$ after sliding one notch: for $1\le i \le k$, we identify~$\gamma^r_i$ with~$\gamma^\ell_{i\pm 1}$, where the sign~$\pm$ is given by~$\eps$, and $i\pm 1$ is understood modulo~$k$.

\item[\textcolor{stepcol}{3. Output}]
	Let~$\tilde\m$ be the resulting map, let $\tilde\bgamma=(\gamma^r_1\sew\gamma^\ell_{1\pm 1},\dots,\gamma^r_k\sew\gamma^\ell_{k\pm 1})$ be the image of~$\bgamma$ in~$\tilde\m$, and let $\tilde\eps\de -$ if $\eps=+$ or $\tilde\eps\de +$ if $\eps=-$. The output of the construction is $\Phi(\m,\bgamma,\eps)\de (\tilde\m,\tilde\bgamma,\tilde\eps)$.
\end{description}

\begin{figure}[ht!]
		\psfrag{l}[B][B][.8]{\textcolor{beta}{$\bgamma$}}
		\psfrag{t}[B][B][.8]{\textcolor{beta}{$\tilde\bgamma$}}
		\psfrag{g}[B][B][.8]{\textcolor{beta}{$\bgamma^\ell$}}
		\psfrag{d}[B][B][.8]{\textcolor{beta}{$\bgamma^r$}}
		\psfrag{1}[B][B][.8]{\textit{slit}}
		\psfrag{2}[B][B][.8][40]{\textit{slide,} $\eps=+$}
		\psfrag{3}[B][B][.8][40]{\textit{slide,} $\eps=-$}
		\psfrag{4}[B][B][.8]{\textit{sew}}
	\centering\includegraphics[width=.95\linewidth]{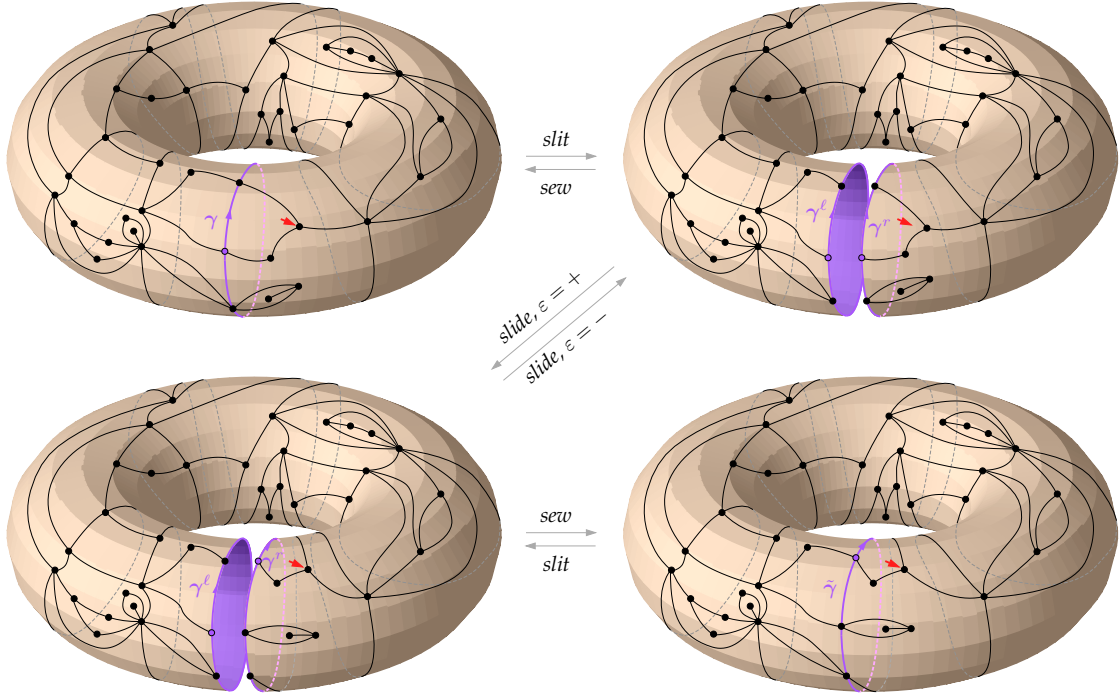}
	\caption{Main involution. The origin of the loops~$\blambda$ and~$\tilde\blambda$ are indicated with purple vertices. When $\eps=+$, we ``screw'' the map one notch along the distinguished loop. We ``unscrew'' one notch along the loop if $\eps=-$.}
	\label{involfig}
\end{figure}

If, furthermore, the faces of~$\m$ are labeled, those of~$\tilde\m$ naturally inherit the labeling from that of~$\m$. In what follows, applying~$\Phi$ to~$\m$ and~$\bgamma$ might be referred to as \emph{rotating}~$\m$ along~$\bgamma$, and the loop~$\bgamma$ will be called the \emph{sliding loop}.

Observe that the action of the $k$-th iterate~$\Phi^k$ is trivial on loops of length~$k$ since its merely consists in a \oldemph{Dehn twist} for these maps: $\Phi^{[\bgamma]}(\m,\bgamma,\eps)=(\m,\bgamma,\pm)$. In particular, $\Phi$ bears no effect on maps whose distinguished loop has length~$1$.

\begin{proposition}\label{propinvol}
The mapping~$\Phi$ is an involution on the set of triples consisting of a map in~$\cM^g$, a distinguished noncontractible simple loop, and a sign in~$\{+,-\}$.

It preserves the number of edges of the map, as well as its type when the faces are labeled.
\end{proposition}

\begin{proof}
During the \hyperlink{slitstep}{\textbf{\textcolor{stepcol}{Slit}}} step, the map~$\m$ gains~$k$ edges, $2$ faces and sees its genus decrease by~$1$. In the \hyperlink{sewstep}{\textbf{\textcolor{stepcol}{Slide and Sew}}} step, it then loses~$k$ edges, $2$ faces and sees its genus increase by~$1$. So the resulting map~$\tilde\m\in\cM^g$ and $|E(\tilde\m)|=|E(\m)|$. Furthermore, the faces are not affected by the process so the type is preserved when the faces are labeled.

The underlying surface and the embedding of~$\bgamma$ are also preserved by~$\Phi$. So~$\tilde\bgamma$ is a noncontractible simple loop of~$\tilde\m$. Let us assume, without loss of generality, that $\eps=+$. In the application of~$\Phi$ to $(\tilde\m,\tilde\bgamma,\tilde\eps)$, the loop~$\tilde\bgamma$ is slit into its left copy $\tilde\bgamma^\ell=(\tilde\gamma^\ell_1,\dots,\tilde\gamma^\ell_k)$ and its right copy $\tilde\bgamma^r=(\tilde\gamma^r_1,\dots,\tilde\gamma^r_k)$, where
$\tilde\gamma^\ell_i=\gamma^\ell_{i+ 1}$ and $\tilde\gamma^r_i=\gamma^r_i$, and then sewn back by identifying $\tilde\gamma^r_i$ with~$\tilde\gamma^\ell_{i- 1}$, that is, $\gamma^r_i$ with~$\gamma^\ell_i$. We thus recover~$\m$, and the distinguished loop is then $(\gamma^r_1\sew\gamma^\ell_{1},\dots,\gamma^r_k\sew\gamma^\ell_{k})=\bgamma$. Finally, the third coordinate is the opposite of~$\tilde\eps$, that is, $\eps$. As a result, $\Phi(\tilde\m,\tilde\bgamma,\tilde\eps)=(\m,\bgamma,\eps)$, as desired.
\end{proof}

\paragraph*{Effects on paths crossing the loop.}\hypertarget{parity}
Rotating along a loop $\bgamma=(\gamma_1,\dots,\gamma_k)$ does not affect the sliding loop itself but creates an ``offset'' for all the paths that traverse it, as is illustrated on Figure~\ref{cross}. More precisely, let us consider a path on the original map~$\m$ of the form $\bp\bullet\bwp\bullet\bp'$, with $[\bp]\ge 1$, $[\bwp]\ge 0$, $[\bp']\ge 1$, and such that~$\bwp$ is empty or part of~$\bgamma$, neither~$\bp$ nor~$\bp'$ uses edges used by~$\bgamma$, and~$\bp$ arrives to the right of~$\bgamma$.

Then, through the rotation along~$\bgamma$, the path $\bp\bullet\bwp\bullet\bp'$ is naturally transformed into $\bp\bullet\tilde\bwp\bullet\bp'$, where~$\tilde\bwp$ is as follows.
\begin{itemize}
	\item If~$\bp'$ leaves from the \oldemph{right} of~$\bgamma$, then $\tilde\bwp\de\varnothing$ if$\bwp=\varnothing$, or $\tilde\bwp\de(\tilde\gamma_r,\dots,\tilde\gamma_s)$ if $\bwp=(\gamma_r,\dots,\gamma_s)$.
	\item If, on the contrary, $\bp'$ leaves from the \oldemph{left} of~$\bgamma$, then
	\begin{itemize}
		\item if $\bwp=\varnothing$, and~$r$ is the index such that~$\gamma_r^-$ is the end of~$\bp$ (and the origin of~$\bp'$),
		\begin{itemize}
			\item $\tilde\bwp\de\bar{{\tilde\gamma}}_{r-1}$ if $\eps=+$,
			\item $\tilde\bwp\de\tilde\gamma_{r}$ if $\eps=-$;
		\end{itemize}
		\item if $\bwp=(\gamma_r,\dots,\gamma_s)$,
		\begin{itemize}
			\item $\tilde\bwp\de(\tilde\gamma_r,\dots,\tilde\gamma_{s-1})$ if $\eps=+$,
			\item $\tilde\bwp\de(\tilde\gamma_r,\dots,\tilde\gamma_{s+1})$ if $\eps=-$.
		\end{itemize}
	\end{itemize}
\end{itemize}

\begin{figure}[ht!]
		\psfrag{v}[Br][Br]{$\bwp=\varnothing$}
		\psfrag{w}[Br][Br]{$\bwp\ne\varnothing$}
		\psfrag{g}[B][B][.8]{\textcolor{beta}{$\bgamma$}}
		\psfrag{h}[B][B][.8]{\textcolor{beta}{$\tilde\bgamma$}}
		\psfrag{b}[B][B][.8]{\textcolor{alpha}{$\bwp$}}
		\psfrag{c}[B][B][.8]{\textcolor{alpha}{$\tilde\bwp$}}
		\psfrag{e}[B][B]{$\eps=+$}
		\psfrag{f}[B][B]{$\eps=-$}
		\psfrag{p}[B][B][.8]{$\bp$}
		\psfrag{q}[B][B][.8]{$\bp'$}
	\centering\includegraphics[width=.95\linewidth]{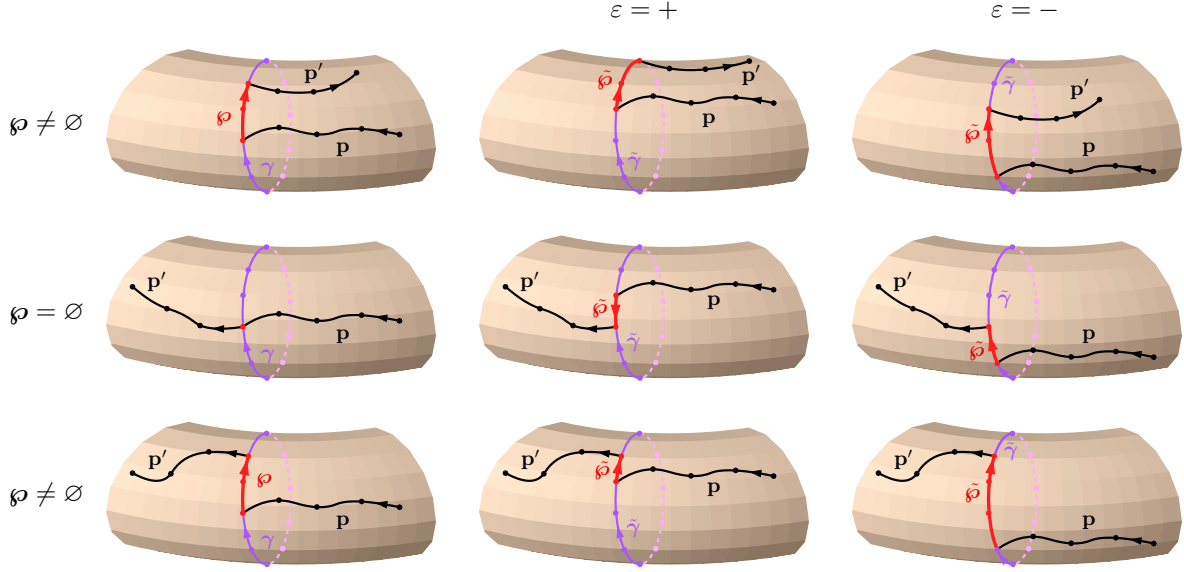}
	\caption{Change of parity in the length of a path crossing the sliding loop.}
	\label{cross}
\end{figure}

We may similarly canonically define the image of paths that cross~$\bgamma$ from left to right or use half-edges of~$\bar\bgamma$. In any case, we see that touching~$\bgamma$ without crossing it does not affect the length of the considered path, while each crossing changes the parity of its length. In particular, a loop homotopic to~$\bgamma$ crosses it an even number of times so the parity of its length is unchanged. On the contrary, the parity of the length of a loop that is not homotopic to~$\bgamma$ may change.

\begin{remark}
Beware that the way we defined here the image of a path crossing the sliding loop is not the only possibility. We could indeed use other parts of the sliding loop for example. Furthermore, the image of a simple path may no longer be simple. If we rotate along a length~$1$ loop for instance, although the involution is the identity in this case, the process described above adds the loop to the path (thus changing the parity of its length) and makes it not simple.
\end{remark}

\subsection{Rightmost loop}\label{secrl}

In order to interpret the desired equations~\eqref{equnicel} and~\eqref{eqcm}, we will apply the involution~$\Phi$ to the desired maps, while canonically defining the distinguished loop along which to rotate. This has to be done in such a way that the image of the chosen loop is recoverable in the resulting map, in order to invert the process.

In this subsection, we will define a first loop and see that everything goes smoothly when rotating along it, even for general maps of any positive genus~$g$. In fact, this first loop is going to be preserved through the involution.

\paragraph*{Rightmost noncontractible simple loop.}\hypertarget{rnslconstr}
Let $\m\in\cM^g$ be a map. We consider the set~$\cS$ of all paths of the form~$\bp\bullet\bgamma$, where~$\bp$ is a simple path starting from the root vertex and~$\bgamma$ is a noncontractible simple loop that intersects~$\bp$ only at its origin. We take the \oldemph{rightmost path}~$\bwp$ of~$\cS$, which might be constructed iteratively as follows.
\begin{itemize}
	\item Starting from the root and turning counterclockwise around the root vertex, we select the first half-edge~$\wp_1$ that belongs to an element of~$\cS$.
	\item Assuming the beginning of the path $(\wp_1,\dots,\wp_i)$ has been constructed:
	\begin{itemize}
		\item if $(\wp_1,\dots,\wp_i) \in \cS$, we stop and set the rightmost path as $\bwp\de(\wp_1,\dots,\wp_i)$;
		\item otherwise, we let $\wp_{i+1}$ be the first half-edge after~$\wp_{i}$, turning counterclockwise around the vertex~$\wp_{i}^+$, that is such that at least one element of~$\cS$ has prefix $(\wp_1,\dots,\wp_{i+1})$.
	\end{itemize}
\end{itemize}
Finally, the \emph{rightmost noncontractible simple loop} is defined as the noncontractible simple loop~$\bgamma$ of the decomposition $\bwp=\bp\bullet\bgamma$ as an element of~$\cS$. We denote this loop by~$\blambda_\m$. See Figure~\ref{rnsl}.

\begin{figure}[ht!]
		\psfrag{l}[B][B][.8]{\textcolor{beta}{$\blambda_\m$}}
		\psfrag{p}[B][B][.8]{$\bp$}
	\centering\includegraphics[width=.5\linewidth]{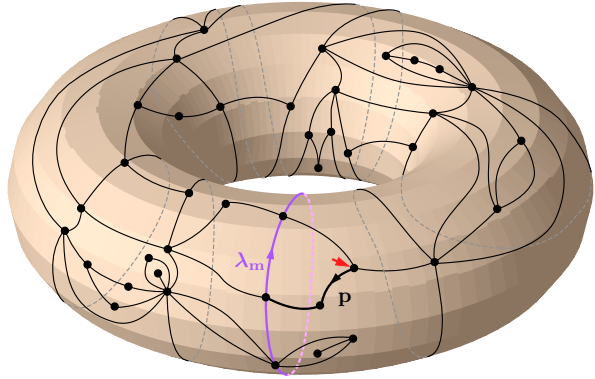}
	\caption{Rightmost noncontractible simple loop of the map from Figure~\ref{map}. The loop~$\blambda_\m$ is the loop part (in purple) of the rightmost path (in thicker lines) from the root consisting of a simple path concatenated with a noncontractible simple loop.}
	\label{rnsl}
\end{figure}

\paragraph*{Rotating along $\blambda_\m$.}
Let us see what happens when we apply the involution~$\Phi$ to the rightmost loop~$\blambda_\m$.

\begin{theorem}\label{thmzero}
Let $g\ge 1$ be fixed. Let us define $\Psi^{\eps}\colon\cM^g\to\cM^g$ by letting $\Psi^{\eps}(\m)$ be the first coordinate of $\Phi(\m,\blambda_\m,\eps)$. Then~$\Psi^{+}$ and~$\Psi^{-}$ are inverse bijections. Furthermore,
\begin{enumerate}[label=(\textit{\roman*})]
	\item they preserve the number of edges of the map;
	\item they preserve the type of the map when the faces are labeled.
\end{enumerate}
\end{theorem}

\begin{proof}
Let $\m\in\cM^g$ and $(\tilde\m,\tilde\blambda,\tilde\eps)=\Phi(\m,\blambda_\m,\eps)$. Since~$\Phi$ is an involution preserving number of edges and types (by Proposition~\ref{propinvol}), it suffices to show that the rightmost loop of~$\m$ is mapped to the rightmost loop of~$\tilde\m$, that is, $\tilde\blambda=\blambda_{\tilde\m}$.

Recalling the construction from the \hyperlink{rnslconstr}{previous paragraph}, let~$\cS$ be the set defined there for the map~$\m$ and~$\tilde\cS$ be the equivalent set defined for~$\tilde\m$. In~$\m$, let also~$\bp$ be the path appearing there, so that $\bp\bullet\blambda_\m$ is the rightmost path of~$\cS$. Since~$\bp$ only intersects~$\blambda_\m$ at its end and reaches it from the right, it is mapped through~$\Phi$ in~$\tilde\m$ into a simple path~$\tilde\bp$ starting from the root vertex, ending at the origin of~$\tilde\blambda$, without intersecting with it otherwise. As a result, $\tilde\bp\bullet\tilde\blambda\in\tilde\cS$.

In order to conclude, it remains to see that it is the rightmost path of~$\tilde\cS$. Arguing by contradiction, let us assume otherwise and let $\tilde\bwp\ne\tilde\bp\bullet\tilde\blambda$ be the rightmost element of~$\tilde\cS$. Then~$\tilde\bwp$ starts with a (possibly null) number of half-edges from~$\tilde\bp\bullet\tilde\blambda$, then leaves it -- from its right since it has to be more to the right than~$\tilde\bp\bullet\tilde\blambda$. If~$\tilde\bwp$ does not intersect with~$\tilde\blambda$ afterward, then this path is completely preserved, when applying the reverse mapping $\Phi(\tilde\m,\tilde\blambda,\tilde\eps)$, into a path in~$\cS$, more to the right than~$\bp\bullet\blambda_\m$, contradicting the definition of~$\bp\bullet\blambda_\m$.

So~$\tilde\bwp$ has to intersect with~$\tilde\blambda$. Let us write $\blambda_\m=(\lambda_1,\dots,\lambda_k)$, $\tilde\blambda=(\tilde\lambda_1,\dots,\tilde\lambda_k)$, and $\tilde\bwp=(\tilde\wp_1,\dots,\tilde\wp_\ell)$. Let~$j$ be the index such that~$\tilde\wp_j^+$ belongs to~$\tilde\blambda$ for the first time after~$\tilde\bwp$ leaves~$\tilde\bp\bullet\tilde\blambda$. Let also~$i$ be the index such that $\tilde\wp_j^+=\tilde\lambda_i^+$. We consider the path $\tilde\bwp_{1\to j}\de (\tilde\wp_1,\dots,\tilde\wp_j)$. This path may possibly use half-edges of~$\tilde\blambda$ if~$\tilde\bwp$ leaves~$\tilde\bp\bullet\tilde\blambda$ after~$\tilde\bp$ reaches~$\tilde\blambda$ but does not cross it, as on the top line of Figure~\ref{cross}. As a result, it is preserved through the reverse mapping and mapped into a path~$\bwp_{1\to j}$. See Figure~\ref{lambdatilde}.

\begin{figure}[ht!]
		\psfrag{p}[B][B][.8]{$\bp$}
		\psfrag{g}[B][B][.8]{\textcolor{beta}{$\blambda_\m$}}
		\psfrag{q}[B][B][.8]{$\tilde\bp$}
		\psfrag{h}[B][B][.8]{\textcolor{beta}{$\tilde\blambda$}}
		\psfrag{m}[B][B][.8]{$\m$}
		\psfrag{n}[B][B][.8]{$\tilde\m$}
		\psfrag{t}[Br][Br][.8]{\textcolor{alpha}{$\tilde\bwp_{1\to j}$}}
	\centering\includegraphics[width=.95\linewidth]{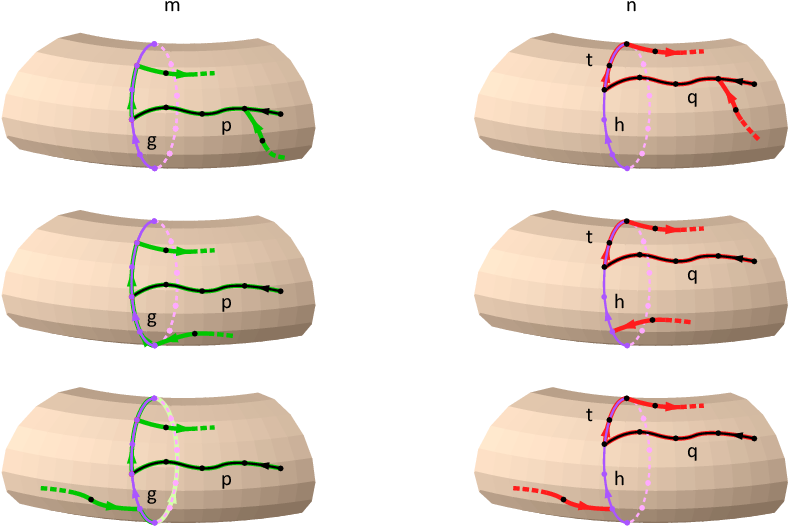}
	\caption{Preservation of the rightmost loop by the involution~$\Phi$. We have, in~$\m$ on the left, the paths~$\bp$ (in black) and~$\blambda_\m$ (in purple), which are mapped through~$\Phi$, in~$\tilde\m$ on the right, respectively into~$\tilde\bp$ (in black) and~$\tilde\blambda$ (in purple). Focusing on the right part of the figure, the rightmost path~$\tilde\bwp$ of~$\tilde\cS$ in~$\tilde\m$ is in red, the part linking the dashes not being represented. It first follows~$\tilde\bp\bullet\tilde\blambda$ on a (possibly null) number of half-edges, then leaves it from its right. After leaving it, it might either not intersect~$\tilde\blambda$ (top line), or reach~$\tilde\blambda$ from its right (middle line), or reach~$\tilde\blambda$ from its left (bottom line). In each case, tracing it back in~$\m$ at the left of the figure provides a path, in green, belonging to~$\cS$ and more to the right than~$\bp\bullet\blambda_\m$.}
	\label{lambdatilde}
\end{figure}

\begin{itemize}
	\item If~$\tilde\bwp_{1\to j}$ reaches~$\tilde\blambda$ from the right, then the path $\bwp_{1\to j}\bullet(\lambda_{i+1},\dots,\lambda_k)\bullet\bar{\bp}$ (with the convention that $(\lambda_{i+1},\dots,\lambda_k)=\varnothing$ if $i=k$), truncated as soon as it becomes nonsimple, belongs to~$\cS$ and is more to the right than~$\bp\bullet\blambda_\m$.
	\item If~$\tilde\bwp_{1\to j}$ reaches~$\tilde\blambda$ from its left, then the path $\bwp_{1\to j}\bullet\overline{(\lambda_{1},\dots,\lambda_{i\pm 1})}\bullet\bar{\bp}$, where the sign~$\pm$ is given by~$\eps$ (with the convention that the middle path is~$\varnothing$ if $i=1$ and $\eps=-$), truncated as soon as it becomes nonsimple, belongs to~$\cS$ and is more to the right than~$\bp\bullet\blambda_\m$.
\end{itemize}
On the example of Figure~\ref{lambdatilde}, the path $\tilde\bwp_{1\to j}$ contains~$\tilde\bp$, so~$\bar{\bp}$ is in the truncated part. Observe that the path in consideration in noncontractible since it crosses exactly once the noncontractible loop~$\blambda_\m$. In both cases, we obtain a contradiction. The result follows.
\end{proof}

\section{Classification of unicellular toroidal maps}

From now on, we restrict our attention to unicellular toroidal maps. In order to derive~\eqref{equnicel}, we will need to rotate along other loops. The very rough idea is that we can rotate along two loops, changing the parity of the one along which we \textbf{do not} rotate. Up to a side factor (the planar term), there are thus 4 ways to go from bipartite to not necessarily bipartite. We will see that it is actually a bit more complicated.

The simpler structure of unicellular toroidal maps will allow us to properly define loops in such a way that rotations along them will be tractable. However, in contrast with the first loop~$\blambda_\m$, the other loops might not be preserved through the involution. As a result, we have to classify more thoroughly the maps.

\subsection{Loops of a unicellular toroidal map}

The structure of a unicellular toroidal map is well known and quite simple. It consists of a \emph{scheme} -- a unicellular toroidal map without vertices of degree~$1$ or~$2$ -- with multiple trees grafted to the left and right of its edges. More precisely, the scheme is obtained by iteratively deleting all the vertices of degree~$1$, and then replacing each maximal chain of degree~$2$ vertices linking vertices of degree at least~$3$ with a single edge. It naturally inherit its root from that of the map; we do not go into details as this will not be used here.

In fact, there are only~$2$ possible schemes in genus~$1$, as we can see in Figure~\ref{loops}: one called \emph{generic} with~$2$ vertices and~$3$ edges, and one called \emph{degenerate} with~$1$ vertex and~$2$ edges. 

\begin{figure}[ht!]
		\psfrag{a}[B][B][.8]{\textcolor{alpha}{$\balpha_\bu$}}
		\psfrag{b}[B][B][.8]{\textcolor{beta}{$\bbeta_\bu$}}
		\psfrag{l}[B][B][.8]{\textcolor{beta}{$\blambda_\bu$}}
		\psfrag{m}[B][B][.8]{\textcolor{mu}{$\bmu_\bu$}}
	\centering\includegraphics[width=.95\linewidth]{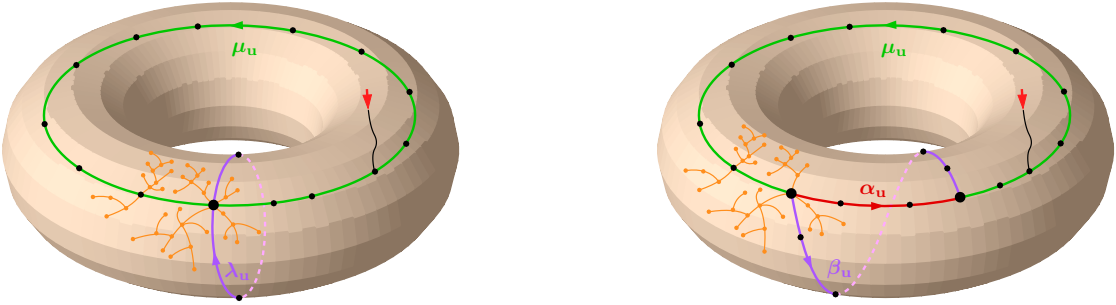}
	\caption{Loops structure in a unicellular toroidal map. Only a few trees are represented (in orange), as well as the path from the root to~$\bmu_\bu$. \textbf{Left.} Case of a degenerate scheme. There are two loops of interest here: $\blambda_\bu$ and~$\bmu_\bu$. \textbf{Right.} Case of a generic scheme. Here, the rightmost loop~$\blambda_\bu$ breaks down as the concatenation of~$\bar\bbeta_\bu$ with~$\balpha_\bu$. We are interested in two extra loops not homotopic to~$\blambda_\bu$, namely $\bmu_\bu\bullet\balpha_\bu$ and $\bmu_\bu\bullet\bbeta_\bu$.}
	\label{loops}
\end{figure}

\paragraph*{Loops.}
The loops are structured as follows. Let $\bu\in\cU^{1}$ be a unicellular toroidal map, and $\blambda_\bu=(\lambda_1,\dots,\lambda_k)$ be its rightmost noncontractible simple loop, defined in Section~\ref{secrl}. There is a unique edge of the scheme that is not visited by~$\blambda_\bu$: it corresponds in~$\bu$ to a path~$\bmu_\bu$ linking~$\lambda_1^-$ to~$\lambda_i^-$ for some $1\le i\le k$.

\begin{itemize}
	\item If $i=1$, then the scheme is degenerate. In this case, the path~$\bmu_\bu$ is actually a noncontractible simple loop, not homotopic to~$\blambda$. We also set in this case $\balpha_\bu\de\varnothing$ to the empty path.

	\item If $i\ge 2$, then the scheme is generic. In this case, we may define two different noncontractible simple loops that are not homotopic to~$\blambda$, by concatenating~$\bmu_\bu$ either with $\balpha_\bu\de (\lambda_i,\dots,\lambda_k)$, or with the reverse~$\bbeta_\bu$ of $(\lambda_1,\dots,\lambda_{i-1})$. 
\end{itemize}

\paragraph*{Partition according to loop parities.}
We partition $\cU^{1}=\cU^{1}_{(0,0)} \sqcup \,\cU^{1}_{(0,1)} \sqcup \,\cU^{1}_{(1,0)} \sqcup \,\cU^{1}_{(1,1)}$ according to the parity of the lengths of the loops $(\blambda_\bu,\bmu_\bu\bullet\balpha_\bu)$, by setting
\[
\cU_{(i,j)}^{1} \de \bigl\{\bu\in\cU^1 \;\big\vert\; \lvert \blambda_\bu \rvert \equiv i \bmod 2 \ \text{ and }\ \lvert\bmu_\bu\bullet\balpha_\bu\rvert\equiv j \bmod 2 \,\bigr\},\qquad 0\le i,j \le 1.
\]

Since, in the generic case, $\lvert \bmu_\bu\bullet\bbeta_\bu \rvert = 2 \lvert \bmu_\bu \rvert + \lvert \blambda_\bu \rvert - \lvert \bmu_\bu\bullet\balpha_\bu \rvert$, it holds that~$\bu$ is bipartite if and only if both loops~$\blambda_\bu$ and~$\bmu_\bu\bullet\balpha_\bu$ have even length, that is, $\cU^{1}_{(0,0)}=\bcU^{1}$. From the observation of the \hyperlink{parity}{last paragraph} of Section~\ref{secinvol}, we also see that Theorem~\ref{thmzero} specializes into bijections
\[
\bcU^{1}=\cU^{1}_{(0,0)}\xrightleftharpoons[\Psi^{\tilde\eps}]{\Psi^{\eps}} \cU^{1}_{(0,1)}
\qquad \eps\in\{+,-\}.\]
As a result, $\bigl\lvert\,\cU^{1}_{(0,1)}\bigr\rvert=\bigl\lvert\,\cU^{1}_{(0,0)}\bigr\rvert=\bigl\lvert\,\bcU^{1}\bigr\rvert$, so that the desired formula~\eqref{equnicel} boils down to
\begin{equation}\label{equnicsple}
\bigl\lvert\,\cU^{1}_{(1,0)}\bigr\rvert + \bigl\lvert\,\cU^{1}_{(1,1)}\bigr\rvert
	= 2\,\bigl\lvert\,\bcU^{1}\bigr\rvert + \bigl\lvert\,\bcU^{0}_{2}\bigr\rvert.
\end{equation}

In order to obtain this equality, we will further partition the four sets involved into several subsets to which we will apply different rotations.

\paragraph*{Extra parameters for refined partitions.}
Rotating along~$\bmu_\bu\bullet\balpha_\bu$ is not as easy as rotating along~$\blambda_\bu$ since it might break the definition of the loops in the resulting map. One thus needs to proceed carefully.

One first parameter to take into account is whether the scheme is generic or degenerate. Another parameter of interest is the location of the root. More precisely, let $\bmu_\bu=(\mu_1,\dots,\mu_q)$. By definition of~$\blambda_\bu$, we see that the root of~$\bu$ is linked by a (possibly empty) path to a vertex~$\mu_j^-$ from the left, for some $1\le j \le q$. We record with the variables 
\[
b_\bu\de \begin{cases}
	1 &\text{ if } j=1\\
	0 &\text{ otherwise}
\end{cases} \qquad\text{ and }\qquad 
e_\bu\de \begin{cases}
	1 &\text{ if } j=q\\
	0 &\text{ otherwise}
\end{cases}
\]
whether this path arrives at the \oldemph{beginning} of~$\bmu_\bu$, and whether it arrives at the \oldemph{end} of~$\bmu_\bu$.

\subsection{Rotating along well-chosen loops}

Using the previously defined parameter, we obtain the following specializations, depicted in Figure~\ref{figspec}. As shorthands, we write ``$\bu \text{ deg.}$'' (resp.~''$\bu \text{ gen.}$'') to mean that the scheme of~$\bu$ is degenerate (resp.\ generic). We add an extra bijection, which is presented below and depicted in Figure~\ref{bijplan}.

\begin{theorem}\label{thmspec}
The following sets are in bijection:
\begin{align}
\Big\{\bu\in\cU^{1}_{(1,0)}\,\colon\, \bu \text{ gen.}\Big\}\hphantom{,\ b_\bu=1}
		& \xrightleftharpoons[\bmu\bullet\balpha,-]{\hspace{1.34em}\rlap{$\scriptstyle\bmu\bullet\balpha,+$}\hspace{3.42em}} 
				\textcolor{blue!50!black}{\Big\{\bu\in\bcU^{1}\,\colon\, e_\bu=0\Big\}}
						\tag{\textit{a}}\label{speca}\\
\Big\{\bu\in\cU^{1}_{(1,0)}\,\colon\, \bu \text{ deg.},\ b_\bu=1 \Big\}
		& \xrightleftharpoons[\bmu\bullet\balpha,-]{\hspace{1.8em}\rlap{$\scriptstyle\bmu,+$}\hspace{2.96em}} 
				\textcolor{blue!50!black}{\Big\{\bu\in\bcU^{1}\,\colon\, e_\bu = 1,\ \lvert \bmu_{\bu}\rvert=1 \Big\}}
						\tag{\textit{b}}\label{specb}\\
\Big\{\bu\in\cU^{1}_{(1,0)}\,\colon\, \bu \text{ deg.},\ b_\bu=0 \Big\}
		& \xrightleftharpoons[\bmu\bullet\bbeta,-]{\hspace{1.8em}\rlap{$\scriptstyle\bmu,+$}\hspace{2.96em}} 
				\textcolor{red!50!black}{\Big\{\bu\in\bcU^{1}\,\colon\, \bu \text{ gen.},\ \lvert\bbeta_\bu\rvert =1 \Big\}}
						\tag{\textit{c}}\label{specc}\\[3mm]
\Big\{\bu\in\cU^{1}_{(1,1)}\,\colon\, \bu \text{ gen.},\ b_\bu=0 \Big\}
		& \xrightleftharpoons[\bmu\bullet\bbeta,-]{\hspace{1.34em}\rlap{$\scriptstyle\bmu\bullet\bbeta,+$}\hspace{3.42em}} 
				\textcolor{red!50!black}{\Big\{\bu\in\bcU^{1}\,\colon\, \bu \text{ gen.},\ \lvert\bbeta_{\bu}\rvert \ge 2 \Big\}}
						\tag{\textit{d}}\label{specd}\\
\Big\{\bu\in\cU^{1}_{(1,1)}\,\colon\, \bu \text{ gen.},\ b_\bu=1 \Big\}
		& \xrightleftharpoons[\bmu\bullet\balpha,-]{\hspace{1.34em}\rlap{$\scriptstyle\bmu\bullet\bbeta,+$}\hspace{3.42em}}
				\textcolor{blue!50!black}{\Big\{\bu\in\bcU^{1}\,\colon\, e_\bu = 1,\ \lvert\bmu_{\bu}\rvert \ge 2 \Big\}}
						\tag{\textit{e}}\label{spece}\\
\Big\{\bu\in\cU^{1}_{(1,1)}\,\colon\, \bu \text{ deg.},\ e_\bu=0 \Big\}
		& \xrightleftharpoons[\bmu\bullet\balpha,+]{\hspace{.28em}\rlap{$\scriptstyle\bmu,-$}\hspace{1.42em}}
					\xrightleftharpoons[\blambda,+]{\hspace{.28em}\rlap{$\scriptstyle\blambda,-$}\hspace{1.42em}}
				\textcolor{red!50!black}{\Big\{\bu\in\bcU^{1}\,\colon\, \bu \text{ deg.}\}}
						\tag{\textit{f}}\label{specf}\\
\Big\{\bu\in\cU^{1}_{(1,1)}\,\colon\, \bu \text{ deg.},\ e_\bu=1 \Big\} 
	& \xrightleftharpoons[]{\hspace{4.76em}} 
			\,\bcU^{0}_{2}.
						\tag{\textit{g}}\label{specg}
\end{align}
\end{theorem}

Except for the last line~\eqref{specg}, the mappings and their inverses are given through the shorthand piece of notation ``$\bgamma,\eps$'', which means the mapping that maps~$\m$ into the first coordinate of $\Phi(\m,\bgamma_\m,\eps)$. 
Equation~\eqref{equnicsple} then follows from the fact that
\begin{itemize}
	\item $\cU^{1}_{(1,0)}$ is partitioned into the 3 first sets on the left,
	\item $\cU^{1}_{(1,1)}$ is partitioned into the 4 last sets on the left,
	\item $\bcU^{1}$ is partitioned into the first, second and fifth sets on the right (in blue), as well as into the third, fourth and sixth sets on the right (in red).
\end{itemize}

The general heuristic behind this theorem is that we want to rotate along a loop different from~$\blambda$. Let us not focus on the degenerate cases~\eqref{specf} and~\eqref{specg} for the moment. When~$\blambda$ has odd length, exactly one of~$\bmu\bullet\balpha$ or~$\bmu\bullet\bbeta$ has even length; so rotating along it is the only possibility to create a bipartite map. In the first 3 cases, $\bu\in\cU^{1}_{(1,0)}$, so we rotate along~$\bmu\bullet\balpha$. In the cases~\eqref{specd} and~\eqref{spece}, $\bu\in\cU^{1}_{(1,1)}$, so we rotate along~$\bmu\bullet\bbeta$. Conversely, when $\bu\in\bcU^{1}$, rotating either along~$\bmu\bullet\balpha$ (cases in blue) or along~$\bmu\bullet\bbeta$ (cases in red) is possible.

\begin{figure}[ht!]
		\psfrag{a}[B][B]{$\xrightleftharpoons[\textcolor{mu}{\bmu}\bullet\textcolor{alpha}{\balpha},-]
						{\textcolor{mu}{\bmu}\bullet\textcolor{alpha}{\balpha},+}$}
		\psfrag{b}[B][B]{$\xrightleftharpoons[\textcolor{mu}{\bmu}\bullet\textcolor{alpha}{\balpha},-]
						{\textcolor{mu}{\bmu},+}$}
		\psfrag{c}[B][B]{$\xrightleftharpoons[\textcolor{mu}{\bmu}\bullet\textcolor{beta}{\bbeta},-]
						{\textcolor{mu}{\bmu},+}$}
		\psfrag{d}[B][B]{$\xrightleftharpoons[\textcolor{mu}{\bmu}\bullet\textcolor{beta}{\bbeta},-]
						{\textcolor{mu}{\bmu}\bullet\textcolor{beta}{\bbeta},+}$}
		\psfrag{e}[B][B]{$\xrightleftharpoons[\textcolor{mu}{\bmu}\bullet\textcolor{alpha}{\balpha},-]
						{\textcolor{mu}{\bmu}\bullet\textcolor{beta}{\bbeta},+}$}
		\psfrag{f}[B][B]{$\xrightleftharpoons[\textcolor{mu}{\bmu}\bullet\textcolor{alpha}{\balpha},+]{\textcolor{mu}{\bmu},-}$}
		\psfrag{g}[B][B]{$\xrightleftharpoons[\textcolor{blue}{\blambda},+]{\textcolor{blue}{\blambda},-}$}
	\centering\includegraphics[width=.95\linewidth]{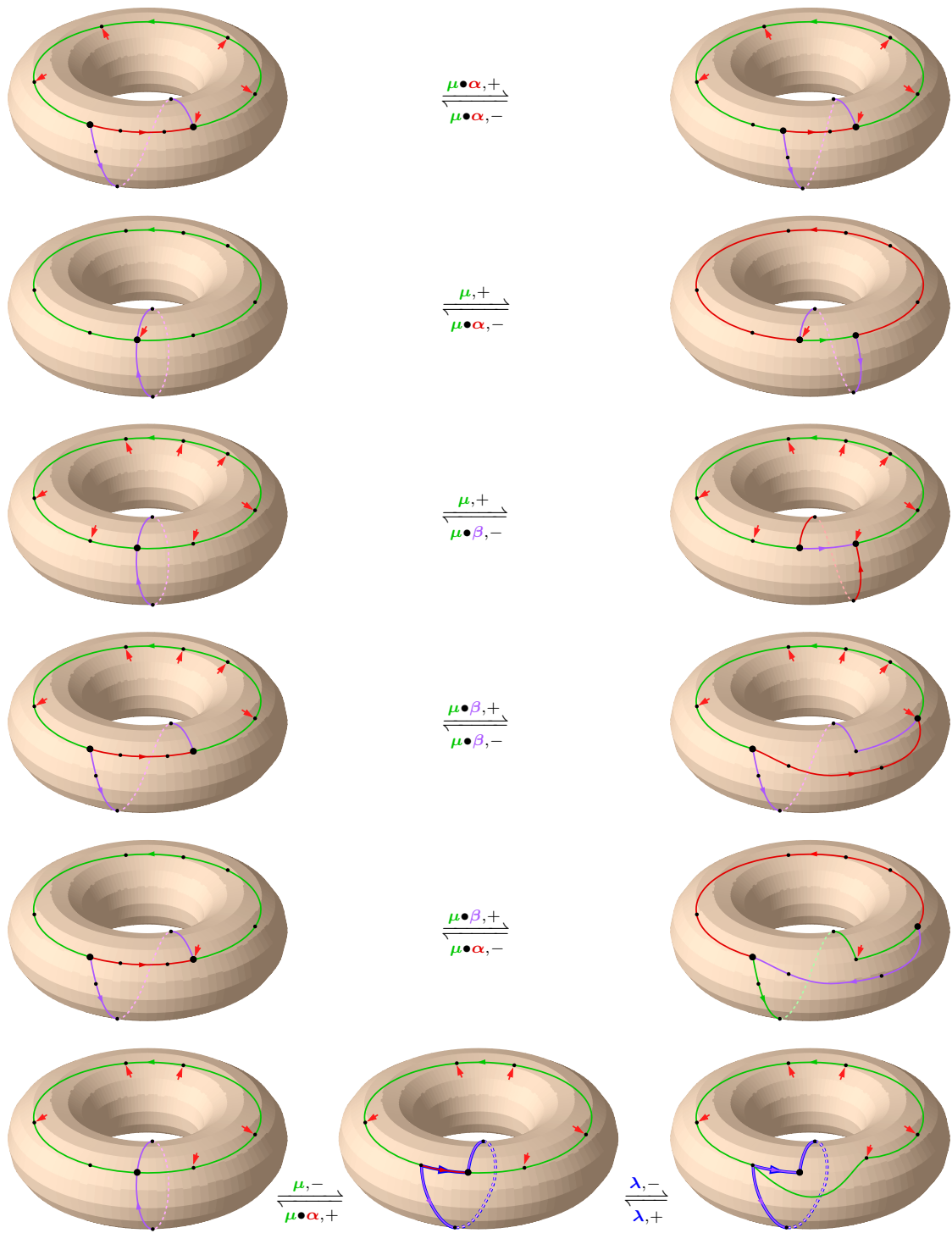}
	\caption{The different specializations of Theorem~\ref{thmspec}. We use the same color code as in Figure~\ref{loops} for all the paths so that we can easily track how they are transformed through the mappings. Only the schemes and the possible location of the path linking to the root are represented.}
	\label{figspec}
\end{figure}

\begin{proof}[Proof of Theorem~\ref{thmspec}, Bijections~\eqref{speca} to~\eqref{specf}]
For bijections~\eqref{speca} to~\eqref{spece}, since we always apply an involution with a~$+$ in one direction and an involution with a~$-$ in the other direction, one only needs to check that the rotations bring us to the claimed sets and that the distinguished loops are the ones used in the other direction.

This is a simple case analysis, which is depicted in Figure~\ref{figspec}. The loop system $\{\bmu,\balpha,\bbeta\}$ is always globally preserved, but permutations within may happen, depending on how the root is connected to the scheme.

\medskip
For~\eqref{specf}, it is similar; we just need to do this twice to take care of the composition.
\end{proof}

\paragraph*{Missing bijection.}\hypertarget{mbpar}
Let us now describe the bijection~\eqref{specg} between $\big\{\bu\in\cU^{1}_{(1,1)}\,\colon\, \bu \text{ deg.},\ e_\bu=1 \big\}$ and the set~$\bcU^{0}_{2}$. Recall that~$\bcU^{0}_{2}$ is the set of plane trees with~$2$ distinct marked vertices at odd distance from the root vertex. These trees thus have a simple underlying structure onto which trees are grafted. This structure consists in a positive even length chain~$\bc$ of edges linking the two marked vertices, and a path~$\bwp$ linking the root vertex to one vertex~$v$ of this chain, in two possible manners:
\begin{enumerate}[label=(\textit{\roman*})]
	\item\label{mb1} either $[\bwp]$ is odd and~$v$ is at even distance to each marked vertex (one of these possibly null);
	\item\label{mb2} or $[\bwp]$ is even (possibly null) and~$v$ is at odd distance to each marked vertex.
\end{enumerate}
We refer to Figure~\ref{bijplan} for visual support.

\paragraph*{From toroidal unicellular maps to trees.}
Let $\bu\in\cU^{1}_{(1,1)}$ have a degenerate scheme and be such that $e_\bu=1$. We slit along the two odd length loops~$\blambda_\bu$ and~$\bmu_\bu$. We thus obtain a topological disk with boundary $\bmu^\ell\bullet\blambda^\ell\bullet\bar\bmu^r\bullet\bar\blambda^r$, where~$\blambda^\ell$ and~$\blambda^r$ (resp.~$\bmu^\ell$ and~$\bmu^r$) are respectively the left and right copies of~$\blambda_\bu$ (resp.\ of~$\bmu_\bu$). 

The root is linked to the boundary by a path~$\bwp$, arriving to the left of~$\bmu^\ell$, at the origin of its last half-edge, since $e_\bu=1$. Let us write $\blambda^\ell=(\lambda_1^\ell,\dots,\lambda_k^\ell)$ and $\blambda^r=(\lambda_1^r,\dots,\lambda_k^r)$. We sew back the boundary in a planar manner as follows.
\begin{itemize}
	\item If $[\bwp]$ is odd, we sew $\bmu^\ell\bullet\blambda^\ell$ onto~$\blambda^r\bullet\bmu^r$.
	\item If $[\bwp]$ is even, we sew $(\bar\lambda_1^r)\bullet\bmu^\ell\bullet(\lambda_1^\ell,\dots,\lambda_{k-1}^\ell)$ onto~$(\lambda_2^r,\dots,\lambda_k^r)\bullet\bmu^r\bullet (\bar\lambda_k^\ell)$.
\end{itemize}
In both cases, we mark the common origin of the sewn paths as~$a$, and their common end as~$b$.

We obtain a plane tree consisting in a simple underlying structure onto which trees are grafted. The structure is made of a positive even length path~$\bp$ linking~$a$ to~$b$, and a path~$\bwp$, arriving to the left of~$\bp$ (possibly at~$a$), linking the root vertex to some vertex~$v$, lying either at positive even distance from~$b$, and (possibly null) even distance from~$a$ if $[\bwp]$ is odd, or at odd distance from~$a$ and~$b$ if~$[\bwp]$ is even. Keeping the resulting plane tree with its marked vertices, we thus obtain an element from~$\bcU^{0}_{2}$.

\begin{figure}[ht!]
		\psfrag{l}[B][B][.8]{\textcolor{beta}{$\blambda_\bu$}}
		\psfrag{c}[B][B][.8]{\textcolor{beta}{$\blambda^\ell$}}
		\psfrag{d}[B][B][.8]{\textcolor{beta}{$\blambda^r$}}
		\psfrag{m}[B][B][.8]{\textcolor{mu}{$\bmu_\bu$}}
		\psfrag{a}[B][B][.8]{\textcolor{mu}{$\bmu^\ell$}}
		\psfrag{b}[B][B][.8]{\textcolor{mu}{$\bmu^r$}}
		\psfrag{x}[B][B][.8]{\textcolor{blue}{$v$}}
		\psfrag{y}[B][B][.8]{\textcolor{blue}{$u$}}
		\psfrag{u}[B][B][.8]{\textcolor{blue}{$a$}}
		\psfrag{v}[B][B][.8]{\textcolor{blue}{$b$}}
		\psfrag{p}[B][B][.8]{\textcolor{blue}{$\bp$}}
		\psfrag{w}[B][B][.8]{$\bwp$}
		\psfrag{h}[B][B][.8]{\textit{slit}}
		\psfrag{i}[B][B][.8][37.35]{\textit{unfold}}
		\psfrag{j}[B][B][.8]{\textit{sew}}
	\centering\includegraphics[width=.95\linewidth]{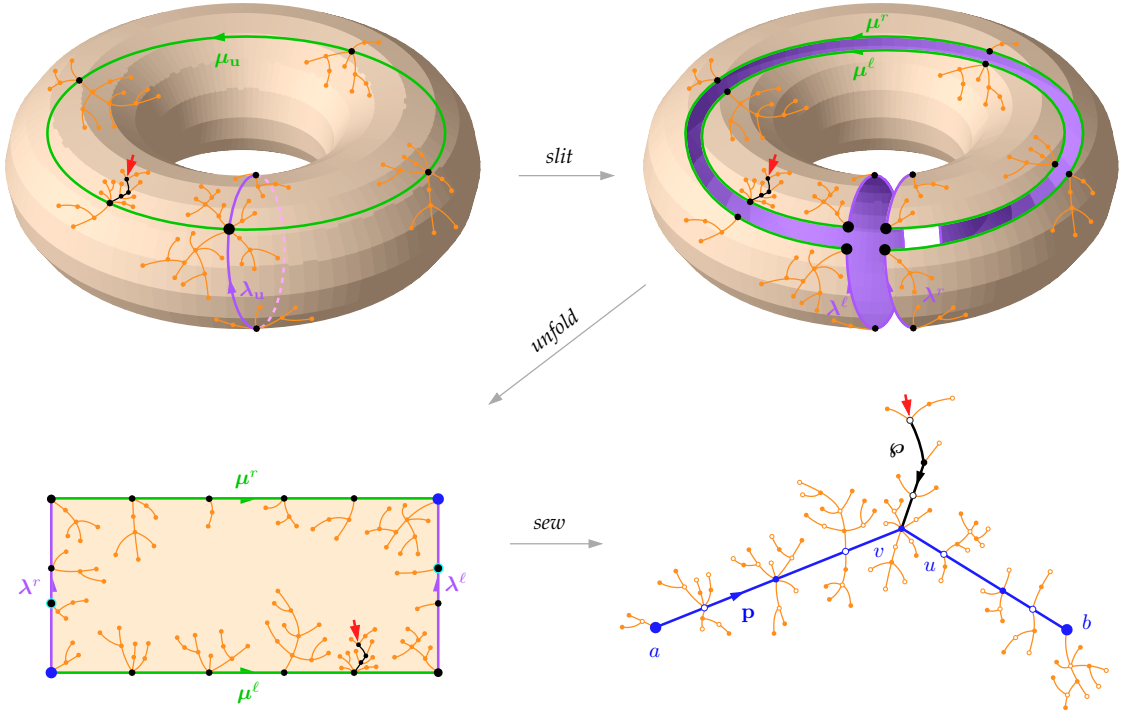}
	\caption{Last bijection of Theorem~\ref{thmspec}. Starting from a unicellular map $\bu\in\cU^{1}_{(1,1)}$ with degenerate scheme and $e_\bu=1$, we slit along the scheme, unfold the resulting rectangle, sew back and mark two vertices. In this example, $[\bwp]=3$ is odd, so we sew the boundary between the vertices in blue. If~$[\bwp]$ were even, we would have sewn between the vertices with a light blue outline instead.}
	\label{bijplan}
\end{figure}

\paragraph*{From trees to toroidal unicellular maps.}
The reverse mapping starts form an element in~$\bcU^{0}_{2}$ and proceeds as follows. Let~$\bc$, $\bwp$ and~$v$ be defined as in the \hyperlink{mbpar}{Missing bijection} paragraph. First, we orient the chain~$\bc$ linking the marked vertices into a path~$\bp$ such that~$\bwp$ arrives to its left. In the eventuality that~$\bwp$ reaches~$\bc$ at one of its extremities, $\bp$ is chosen in such a way that this vertex is its origin (note that, in this case, $[\bwp]$ is odd so that we are in case~\ref{mb1}). We let~$u$ be the vertex right after~$v$ along~$\bp$.

We slit along~$\bp$ and obtain a topological disk whose boundary consists of the left copy~$\bp^\ell$ and right copy~$\bp^r$ of~$\bp$. On each copy, there is also a copy of~$u$: we let~$u^\ell$ be the copy on~$\bp^\ell$.
\begin{itemize}
	\item If $[\bwp]$ is odd, we split~$\bp^\ell$ at~$u^\ell$ into $\bmu^\ell\bullet\blambda^\ell$. We also split~$\bp^r$ into $\blambda^r\bullet\bmu^r$ in such a way that $[\blambda^r]=[\blambda^\ell]$, and thus that $[\bmu^r]=[\bmu^\ell]$.
	\item If $[\bwp]$ is even, we let~$\bmu^\ell$ be the part of~$\bp^\ell$ starting at its second half-edge and ending at~$u^\ell$, and we let~$\blambda^\ell$ be the part of~$\bp^\ell$ between~$u^\ell$ and its end, concatenated with the first subsequent half-edge on the boundary. We then define~$\blambda^r$ and~$\bmu^r$ in such a way that the boundary is given by the cycle~$\bmu^\ell\bullet\blambda^\ell\bullet\bar\bmu^r\bullet\bar\blambda^r$ and $[\blambda^r]=[\blambda^\ell]$.
\end{itemize}
We sew~$\bmu^\ell$ onto~$\bmu^r$, as well as~$\blambda^\ell$ onto~$\blambda^r$.

\begin{proof}[Proof of Theorem~\ref{thmspec}, Bijection~\eqref{specg}]
Let us denote the above mappings, depicted in Figure~\ref{bijplan}, by
\begin{align}
\Big\{\bu\in\cU^{1}_{(1,1)}\,\colon\, \bu \text{ deg.},\ e_\bu=1 \Big\} 
	& \xrightleftharpoons[\hspace{1em}\Theta\hspace{1em}]{\Xi} \,\bcU^{0}_{2}.\tag{\ref{specg}}
\end{align}

Let $\bu$ be in the left-hand side set. Since $[\bp]=[\blambda_\bu]+[\bmu_\bu]$ and since the loops~$\blambda_\bu$ and~$\bmu_\bu$ have odd length, we see that~$\bp$ has even length, larger than~$2$. In particular, $a$ and~$b$ are distinct. Furthermore, the choice of~$a$, $b$ depending on the parity of the length of~$\bwp$ guarantees that~$a$ and~$b$ lie at odd distance from the root vertex. As a result, $\Xi(\bu)\in\bcU^{0}_{2}$. 

Conversely, let $\bt\in\bcU^{0}_{2}$ and let us look at $\Theta(\bt)$. We denote the origin and end of~$\bp$ by~$a$ and~$b$. Since~$a$ and~$b$ lie at odd distance from the root vertex, they lie at even distance from each other. The sewings of~$\bmu^\ell$ onto~$\bmu^r$ and of~$\blambda^\ell$ onto~$\blambda^r$ yields a toroidal map $\bu=\Theta(\bt)$ with degenerate scheme given by the two loops $\bmu^\ell\sew\bmu^r$ and~$\blambda^\ell\sew\blambda^r$, having both odd length. Since~$\bwp$ arrives on~$\bmu^\ell$ before its end, we see that~$\blambda_\bu=\blambda^\ell\sew\blambda^r$ and that $\bmu_\bu=\bmu^\ell\sew\bmu^r$. Finally, since~$\bwp$ arrives on~$\bmu^\ell$ exactly one edge before its end, we have $e_\bu=1$, so that $\Theta(\bt)$ belongs to the set we claim.

From the observations above, it should be clear that~$\Xi$ and~$\Theta$ are inverse one from the other.
\end{proof}

\section{Covered maps}\label{seccov}

Let us finally explain how to derive~\eqref{eqcm}. We may apply the involution~$\Phi$ on covered maps when the marked loop is included in the covering map as follows. Let $(\m,\bu)$ be a covered map of some genus $g\ge 1$, let~$\bgamma$ be a noncontractible simple loop \textbf{in the covering map}~$\bu$, and let $\eps\in\{+,-\}$. We can apply~$\Phi$ to~$\m$ and~$\bgamma$, setting $(\tilde\m,\tilde\bgamma,\tilde\eps)\de\Phi(\m,\bgamma,\eps)$. 

Now, an edge of~$\bu$ is 
\begin{itemize}
	\item either on~$\bgamma$ and is then transported by~$\Phi$ to an edge of~$\tilde\bgamma$ in~$\tilde\m$;
	\item or not on~$\bgamma$ and preserved by~$\Phi$ in~$\tilde\m$.
\end{itemize}
From this, we observe that the image~$\tilde\bu$ of~$\bu$ gives a covered map $(\tilde\m,\tilde\bu)$ and that $(\tilde\bu,\tilde\bgamma,\tilde\eps)=\Phi(\bu,\bgamma,\eps)$. As a result, we obtain the following corollary to Proposition~\ref{propinvol}.

\begin{corollary}\label{corcov}
The mapping~$\Phi$ extends to an involution on the set of triples $((\m,\bu),\bgamma,\eps)$ consisting of a covered map $(\m,\bu)$, a distinguished noncontractible simple loop~$\bgamma$ \textbf{of the covering map}~$\bu$, and a sign $\eps\in\{+,-\}$.

If $((\tilde\m,\tilde\bu),\tilde\bgamma,\tilde\eps)\de\Phi((\m,\bu),\bgamma,\eps)$, then $(\tilde\m,\tilde\bgamma,\tilde\eps)=\Phi(\m,\bgamma,\eps)$ and $(\tilde\bu,\tilde\bgamma,\tilde\eps)=\Phi(\bu,\bgamma,\eps)$.
 
Furthermore, it preserves the genus, the number of edges of both the ambient map and of the covering map, as well as the type when the faces are labeled.
\end{corollary}

Recall that Equation~\eqref{equnicel} was obtained by partitioning $\cU^{1}=\cU^{1}_{(0,0)} \sqcup \,\cU^{1}_{(0,1)} \sqcup \,\cU^{1}_{(1,0)} \sqcup \,\cU^{1}_{(1,1)}$ and relating these four subsets to~$\bcU^{1}$ and, subsidiarily, to~$\bcU^{0}_{2}$. Let us then see how this partition translates, when applied to the covering map.

First of all, we observe that, in any genus $g\ge 0$, the bipartiteness of the ambient map is given by that of the covering map.

\begin{proposition}\label{propbipcov}
The ambient map of a covered map of any genus, with even degree faces, is bipartite if and only if its covering map is bipartite.
\end{proposition}

\begin{proof}
Let us call \emph{$2$-coloring} of a map a proper black and white coloring of its vertices, with the root vertex colored white. A map is then bipartite if and only if it admits a $2$-coloring, and, in this case, the $2$-coloring is unique.

\medskip
Let $(\m,\bu)$ be a covered map with even degree faces. If there exists a $2$-coloring of~$\m$, then this coloring immediately yields a $2$-coloring of~$\bu$, since its root vertex is the same as~$\m$ and its edges are all edges of~$\m$.

\medskip
Conversely, let us assume that~$\bu$ admits a $2$-coloring, and let~$n$ be the number of edges of~$\bu$. Slitting the map~$\m$ along~$\bu$ yields a map on a topological disk whose length~$2n$ boundary inherits the $2$-coloring from~$\bu$. This map is plane and its faces have the same degrees as~$\m$, except the extra (external) face, which has degree~$2n$. Since it is a plane map with even degree faces, it admits\footnote{Recall that a plane with even degree faces is bipartite. Indeed, if it has an odd length cycle, cutting along it yields a map with exactly one odd degree face, which contradicts the fact that the sum of the face degrees is twice the number of edges.} a $2$-coloring. And since this $2$-coloring is unique, it has to be compatible with that of the boundary inherited from~$\bu$. Recovering~$\m$ by sewing back the boundary of the disk as it was initially, we obtain a $2$-coloring of~$\m$. Note that the coloring is compatible with the sewing since the glued vertices originally had the same color.
\end{proof}

Let $\ba=(a_1,\ldots,a_r)$ be a tuple of positive even integers, and $n\de \frac12 \sum_{j=1}^r a_j -r+1$ be the number\footnote{The ambient map has $\frac12 \sum_{j=1}^r a_j$ edges and~$r$ faces, thus $\frac12 \sum_{j=1}^r a_j-r+2-2g$ vertices, by the Euler characteristic formula. The covering map has as many vertices and~$1$ face, thus~$n$ edges by the Euler characteristic formula.} of edges of the covering map of any element in~$\bccM^{g}_{2i}(\ba)$ for any $g\ge 0$, $i\ge 0$. We thus have that $\ccM^{1}(\ba)$ is the set of covered maps of type~$\ba$ whose covering map lies in~$\cU^1$. From Proposition~\ref{propbipcov}, we furthermore obtain that $\bccM^{1}(\ba)$ is the subset of~$\ccM^{1}(\ba)$ whose covering map belongs to~$\bcU^{1}$ and that~$\bccM_{2}^{0}(\ba)$ is the set of plane covered maps of type~$\ba$ whose covering tree lies in~$\bcU^{0}_{2}$.

We partition $\ccM^{1}(\ba)$ according to the partition of~$\cU^1$ into nine subsets used above: $\cU^{1}_{(1,0)}$, $\cU^{1}_{(1,1)}$, and the seven left-hand sides of Theorem~\ref{thmspec}. Corollary~\ref{corcov} ensures that, except for the last one, \eqref{specg}, all these subsets are in bijection with $\bccM^{1}(\ba)$ or a subset thereof, as prescribed by the right-hand sides of Theorem~\ref{thmspec}.

In order to conclude the proof of Proposition~\ref{propcm}, it remains to see that the last subset,
\[\big\{(\m,\bu)\in \ccM^{1}(\ba) \,\colon\, \bu\in\cU^{1}_{(1,1)},\, \bu \text{ deg.},\ e_\bu=1 \big\},\]
is in bijection with $\bccM_{2}^{0}(\ba)$. To obtain this, it suffices to see that the bijections~$\Xi$ and~$\Theta$ interpreting~\eqref{specg} extend to bijections between the corresponding sets of covered maps. 

These bijections consists in slitting along either the degenerate scheme in genus~$1$, or along a Y-shape in genus~$0$ in order to obtain a topological disk, and then sew it back along the other structure. Similarly as above, we can apply these mappings to the covering map. The extra vertices and edges of the ambient map do not interfere with the constructions and the faces are naturally preserved in the processes. The type is thus preserved and the mappings easily extend to bijections.

\bibliographystyle{alpha}
\bibliography{main}

\end{document}